\documentclass[a4, 12pt]{amsart}
\usepackage{amssymb}
\usepackage{amstext}
\usepackage{amsmath}
\usepackage{amscd}
\usepackage{latexsym}
\usepackage{amsfonts}
\usepackage{color}
\usepackage{mathrsfs}
\usepackage{enumerate}
\usepackage{textcomp}
\usepackage[all]{xy}

\usepackage{appendix}

\theoremstyle{plain}
\newtheorem{thm}{Theorem}[section]
\newtheorem*{thm*}{Theorem}
\newtheorem*{cor*}{Corollary}
\newtheorem*{defn*}{Definition}

\newtheorem{prop}[thm]{Proposition}
\newtheorem{lem}[thm]{Lemma}
\newtheorem{cor}[thm]{Corollary}
\newtheorem{claim}[thm]{Claim}
\newtheorem*{claim*}{Claim}
\newtheorem*{ac}{Acknowledgments}

\theoremstyle{definition}

\newtheorem{ex}[thm]{Example}
\newtheorem{rem}[thm]{Remark}

\theoremstyle{remark}

\numberwithin{equation}{thm}

\def\rank{\mathrm{rank}}
\def\a{\mathfrak a}
\def\b{\mathfrak b}

\def\e{\mathrm{e}}
\def\m{\mathfrak m}

\def\p{\mathfrak p}

\def\P{\mathfrak P}

\newcommand{\rma}{\mathrm{a}}

\newcommand{\rme}{\mathrm{e}}

\newcommand{\rmr}{\mathrm{r}}

\newcommand{\rmR}{\mathrm{R}}
\newcommand{\rmS}{\mathrm{S}}

\newcommand{\calR}{\mathcal{R}}

\def\depth{\mathrm{depth}}
\def\Supp{\mathrm{Supp}}

\def\Ass{\mathrm{Ass}}

\def\Min{\mathrm{Min}}

\def\height{\mathrm{ht}}

\def\Spec{\mathrm{Spec}}

\begin{document}

\setlength{\baselineskip}{17pt}
\title{Normal Sally modules of rank one}
\author{Tran Thi Phuong}
\address{Faculty of Mathematics ans Statistics, ton Duc Thang University, Ho CHi Minh City, Vietnam}
\email{tranthiphuong@tdt.edu.vn}


\thanks{The author is partially supported by JSPS KAKENHI 26400054.}

\begin{abstract}In this paper, we explore the structure of the normal Sally modules of rank one with respect to an $\m$-primary ideal in a Nagata reduced local ring $R$ which is not necessary Cohen-Macaulay. As an application of this result, when the base ring is Cohen-Macaulay analytically unramified, the extremal bound on the first normal Hilbert coefficient leads to the depth of the associated graded rings $\overline{\mathcal G}$ with respect to a normal filtration is at least $\dim R-1$ and $\overline{\mathcal G}$ turns in to Cohen-Macaulay when the third normal Hilbert coefficient is vanished.
\end{abstract}

\dedicatory{Dedicated to Professor Shiro Goto on the occasion
of his $70^{th}$ birthday}

\maketitle
\tableofcontents
\section{Introduction}
Throughout this paper, let $R$ be an analytically unramified Noetherian local ring with the maximal ideal $\m$ and $d=\dim R >0$. Let $I$ be an $\m$-primary ideal of $R$ and suppose that $I$ contains a parameter ideal $Q=(a_1, a_2,...,a_d)$ of $R$ as a reduction. Let $\ell_R(M)$ denote the length of an $R$-module $M$ and $\overline{I^{n+1}}$ denote the integral closure of $I^{n+1}$ for each $n\geq 0$. Since $R$ is an analytically unramified, there are integers $\{\overline{\e_i}(I)\}_{0\leq i \leq d}$ such that the equality 
$$\ell_R(R/\overline{I^{n+1}}) = \overline{ \e_0}(I) {{n+d}\choose{d}} - \overline{ \e_1}(I) {{n+d-1}\choose{d-1}} + ... +(-1)^d\overline{ \e_d}(I)$$
holds true for all integers $n\gg 0$, which we call the normal Hilbert coefficients of $R$ with respect to $I$. We will denote by $\{\e_i(I)\}_{0\leq i \leq d}$ the ordinary Hilbert coefficients of $R$ with respect to $I$. Let 
$$\mathcal R = \rmR(I): = R[It] \text{~and~} T = \rmR(Q):=R[Qt] \subseteq R[t]$$
denote, respectively, the Rees algebra of $I$ and $Q$, where $t$ stands for an indeterminate over $R$. Let 
$$\mathcal R^{'} = \rmR^{'}(I):=R[It, t^{-1}]  {\text{~and~}} \mathcal G= \mathcal G(I):= \mathcal R^{'}/t^{-1}\mathcal R^{'} \cong \oplus_{n\ge 0} I^n/I^{n+1}$$
denote, respectively, the extended Rees algebra of $I$ and the associated graded ring of $R$ with respect to $I$.
Let $\overline{\mathcal R}$ denote the integral closure of $\mathcal R$ in $R[t]$ and $\overline{\mathcal G} = \oplus_{n\ge 0} \overline{I^n}/\overline{I^{n+1}}$ denote the associated graded ring of the normal filtration $\{\overline{I^n}\}_{n \in \mathbb Z}$. Then $\overline{\mathcal R} = \oplus_{n\ge 0} \overline{I^n}t^n$ and $\overline{\mathcal R}$ is a module-finite extension of $\mathcal R$ since $R$ is analytically unramified (see \cite[Corollary 9.2.1]{SH}). For the reduction $Q$ of $I$, the reduction number of $\{\overline{I^n}\}_{n \in \mathbb Z}$ with respect to $Q$ is defined by 
$$\rmr_Q(\{\overline{I^n}\}_{n\in \mathbb Z}) = \min\{r\in \mathbb Z \mid \overline{I^{n+1}} = Q\overline{I^n}, \forall n\ge r\}.$$
 
The notion of Sally modules of normal filtrations was introduced by \cite{CPR2} in order to find the relationship between a bound on the first normal Hilbert coefficients $\overline{\e_1}(I)$ and the depth of $\overline{\mathcal G}$ when $R$ in an analytically unramified Cohen-Macaulay rings $R$. Following \cite{CPR2}, we generalize the definition of normal Sally modules to the non-Cohen-Macaulay cases, and we define the normal Sally modules $\overline{S} = \overline S_Q(I)$ of $I$ with respect to a minimal reduction $Q$ to be the cokernel of the following exact sequence 
$$0\longrightarrow \overline I T \longrightarrow \overline{\mathcal R}_+(1) \longrightarrow \overline S \longrightarrow 0 $$
of graded $T$-modules. Since $\overline{\mathcal R}$ is a finitely generated $T$-module, so is $\overline S$ and we get 
$$\overline S = \oplus_{n\geq 1} \overline{I^{n+1}}/Q^n\overline{I}$$
by the following isomorphism 
$$\overline{\mathcal R}_+(1) \overset{t^{-1}}{\stackrel{\sim}{\longrightarrow}} \sum_{n\geq 0}\overline{I^{n+1}}t^n (\supseteq \sum_{n\geq 0}(Q^{n}\overline I)t^n = \overline IT)$$
of graded $T$-modules.

To state the results of this paper, let us consider the following four conditions:\\
$(C_0)$ The sequence $a_1, a_2,...,a_d$ is a $d$-sequence in $R$ in the sense of \cite{H}\\ 
$(C_1)$ The sequence $a_1, a_2,...,a_d$ is a $d^+$-sequence in $R$, that is for all integers $n_1, n_2,...n_d\ge 1$ the sequence $a_1^{n_1}, a_2^{n_2},...,a_d^{n_d}$ forms a $d$-sequence in any order. \\
$(C_2)$ $(a_1,a_2,...,\check{a_i},...,a_d):_R a_i \subseteq I$ for all $1\leq i \leq d$\\
$(C_3)$ $\depth R> 0$ and $\depth R > 1$ if $d\ge 2$.  
 
These conditions $(C_0)$, $(C_1)$, and $(C_2)$ are exactly the same as in \cite{GO}. The conditions $(C_1)$, $(C_2)$, and $(C_3)$ are automatically satisfied if $R$ is Cohen-Macaulay. Conditions $(C_1)$ and $(C_3)$ imply the ring $R$ has the property $(S_2)$. 

The main result of this research is as follows.
\begin{thm}\label{mainth} Let $R$ be a Nagata and reduced local ring with the maximal ideal $\m$ and $d=\dim R >0$. Let $I$ be an $\m$-primary ideal of $R$ and suppose that $I$ contains a parameter ideal $Q$ of $R$ as a reduction. Assume that conditions $(C_1)$, $(C_2)$, and $(C_3)$ are satisfied. Then the followings are equivalent to each other.  
\begin{enumerate}[\rm (1)]
\item $\overline{\e_1}(I) = \e_0(I) + \e_1(Q) - \ell_R(R/\overline{I}) + 1$.
\item $\m\overline{S} = (0)$ and $\rank_B\overline{S} = 1$, where $B = T/\m T$.
\item $\overline S \cong B(-q)$ as graded $T$-modules for some integer $q\ge 1$.
\end{enumerate}
When this is the case 
\begin{enumerate}[\rm (a)]
\item $\overline S $ is a Cohen-Macaulay $T$-module.
\item Put $t = \depth R$. Then 
$$
 \depth \overline{\mathcal G} \ge \left\{
 \begin{array}{ll}
 d-1 & ~\text{if}~~t \ge d-1,\\
 t & ~\text{if}~~t < d-1.
 \end{array}
 \right.
 $$
\item for all $n\ge 0$,
\begin{eqnarray*}
\ell_R(R/\overline{I^{n+1}})&=& \e_0(I) {{n+d}\choose{d}} - \{\e_0(I) + \e_1(Q) - \ell_R(R/\overline{I})\} {{n+d-1}\choose{d-1}} \\
 & & + \sum_{i=2}^{d}(-1)^i\{\e_{i-1}(Q) + \e_i(Q)\}{{n+d-i}\choose{d-i}}
\end{eqnarray*} 
if $n < q$, and 
\begin{eqnarray*}
\ell_R(R/\overline{I^{n+1}})&=& \e_0(I) {{n+d}\choose{d}} - \{\e_0(I) + \e_1(Q) - \ell_R(R/\overline{I}) + 1\} {{n+d-1}\choose{d-1}} \\
 & & + \sum_{i=2}^{d}(-1)^i\{\e_{i-1}(Q) + \e_i(Q) +{{q}\choose{i-1}}\}{{n+d-i}\choose{d-i}}
\end{eqnarray*}  
if $n\ge q$.
Hence $\overline{\e_i}(I) = \e_{i-1}(Q) + \e_i(Q)+{{q}\choose{i-1}}$ for all $2\leq i\leq d$.
\end{enumerate}
\end{thm}

The relationship between the equality $\overline{\e_1}(I) = \e_0(I) - \ell_R(R/\overline{I}) + 1$ and the depth of $\overline {\mathcal G}$ in an analytically unramified Cohen-Macaulay local ring was examined in \cite{CPR2}. In their paper, they proved that if $R$ is an analytically unramified Cohen-Macaulay ring possessing a canonical module $\omega_{R}$, then $\overline{\e_1}(I) = \e_0(I) - \ell_R(R/\overline{I}) + 1$ makes $\depth \overline{\mathcal G}\ge d-1$ (\cite[Theorem 2.6]{CPR2}). Moreover if $d\ge 3$ and $\overline{\e_3}(I)=0$, then this equality $\overline{\e_1}(I) = \e_0(I) - \ell_R(R/\overline{I}) + 1$ leads to the Cohen-Macaulayness of $\overline{\mathcal G}$ ((\cite[Proposition 3.4]{CPR2})). The assumption $R$ has a canonical module assures that $\overline{\mathcal R}$ satisfies the Serre condition $(S_2)$ as a ring, which is essential for the proofs of their results. In this paper, as an application of Theorem \ref{mainth}, we will prove that the above results \cite[Theorem 2.6 and Proposition 3.4]{CPR2} still hold true even when we delete the assumption that the base ring possessing a canonical module $\omega_{R}$, as stated in the following.

\begin{thm}\label{e2} Let $R$ be a analytically unramified Cohen-Macaulay local ring with the maximal ideal $\m$, dimension $d=\dim R >0$, and $I$ an $\m$-primary ideal of $R$ containing a parameter ideal $Q$ of $R$ as a reduction. Assume that $\overline{\e_1}(I) = \e_0(I) - \ell_R(R/\overline{I}) + 1$. Then the following assertions hold true.
\begin{enumerate}[\rm(1)]
\item $\depth \overline{\mathcal G} \ge d-1$. 
\item If $d\ge 3$ and $\overline{\e_3}(I) = 0$ then $\overline {\mathcal G}$ is Cohen-Macaulay, $\overline{\e_2}(I) = 1$, and the normal filtration has reduction number two. 
\end{enumerate}
\end{thm}

Now it is a position to explain how this paper is organized. This paper contains of 3 sections. The introduction part is this present section. In Section 2 we will collect some auxiliary results on normal Sally modules and normal Hilbert functions. We will prove Theorem \ref{mainth}, Theorem \ref{e2}, and explore a consequence of Theorem \ref{mainth} in the Cohen-Macaulay case in Section 3. 

\section{Auxiliaries}
In this section we will collect properties of the normal Sally modules and the normal Hilbert coefficients which are essential for the proof of our main results. Throughout this section, let $R$ be an analytically unramified Noetherian local ring with the maximal ideal $\m$ and $\dim R = d\ge 1$. Let $I$ be an $\m$-primary ideal and assume that $I$ contains a minimal reduction $Q = (a_1,...,a_d)$.

Let us begin with the following lemma which play an important role on computing the normal Hilbert functions and on examining the structure of $\overline S$.
 
\begin{lem} \label{Aux} {\cite[Lemma 2.1]{GO}} Suppose that conditions $(C_0)$ and $(C_2)$ are satisfied. Then 
$$T/\overline I T \cong (R/\overline I)[X_1,X_2,...,X_d]$$ 
as graded $R$-algebras, where $(R/\overline I)[X_1,X_2,...,X_d]$ denotes the plynomial ring with $d$ indeterminates over Artinian local ring $R/\overline I$. Hence $T/\overline I T$ is a Cohen-Macaulay ring with $\dim T/\overline I T = d$.
\end{lem}

Under our generalized assumption, the results \cite[Proposition 2.2]{CPR2} on the set of associated prime ideals and the dimension of $\overline S$ do not change, and moreover we obtain a formula on $\depth \overline{\mathcal G}$ as follows.

\begin{lem}\label{Sallymodule} The following assertions hold true.
\begin{enumerate}[\rm (1)]
\item {\cite[Lemma 2.1]{GO}} $\m^{\ell}\overline S = 0$ for integer $\ell \gg 0$. Hence $\dim_{T}\overline S \le d$.
\item {\cite[Lemma 2.3]{GO}} Suppose that conditions $(C_0)$, $(C_2)$ and $(C_3)$ are satisfied and $T$ is a $(S_2)$ ring. Then $\Ass_T(\overline S) \subseteq \{\m T\}$. Hence $\dim_T\overline S = d$ provided $\overline S \ne (0)$.
\item Suppose that conditions $(C_0)$, $(C_2)$ and $(C_3)$ are satisfied. Then $\depth \overline{\mathcal G} \ge \depth R$ if $\overline S = (0)$ and 
$$
 \depth \overline{\mathcal G} \ge \left\{
 \begin{array}{ll}
 \depth R & (\depth_T \overline S \ge \depth R+1),\\
 \depth_T \overline S -1  & (\depth_T \overline S \le \depth R ~\text{or} ~ \depth R = d-1)
 \end{array}
 \right.
 $$
if $\overline S \ne (0)$. 
\end{enumerate}
\end{lem}

\begin{proof} The proof of (1) is the same as that of {\cite[Lemma 2.1]{GNO}}. The proof of (2) is almost the same as that of {\cite[Lemma 2.3]{GO}}, let us include a proof for the sake of completeness. We may assume that $\overline S \ne (0)$. Let $P\in \Ass_{T} \overline S$. Then $\m T \subseteq P$. Assume that $P\ne \m T$. Then $\height_{T}P\ge 2$ since $\height_{T} \m T = 1$. Therefore $\depth T \ge 2$ by condition $(S_2)$. Now we consider the following exact sequences
$$0\longrightarrow (\overline{I}T)_P \longrightarrow (\overline{\mathcal R}_+(1))_P\longrightarrow \overline S_P\longrightarrow 0$$
and 
$$0\longrightarrow (\overline{I}T)_P \longrightarrow T_P \longrightarrow T_P/(\overline{I}T)_P \longrightarrow 0$$
of graded $T_P$-modules. Since $\depth_{T_P}(\overline{\mathcal R}_+(1))_P\ge 1$ and $\depth_{T_P}\overline S_P = 0$, $\depth_{T_P}(\overline{I}T)_P = 1$. Therefore $\depth_{T_P}T_P/(\overline{I}T)_P = 0$ by the second exact sequence. Moreover since conditions $(C_0)$ and $(C_2)$ are satisfied, $T/\overline{I}T$ is a Cohen-Macaulay ring by \cite[Proposition 2.2]{GO}, so is $T_P/(\overline{I}T)_P$. Therefore $P\in \Min_{T}T_P/(\overline{I}T)_P = \{\m T\}$, which is a contradiction. Thus $P=\m T$ as desired.   

The statement (3) follows by comparing depths of $T$-modules in the following exact sequences 
$$(1)~~~~~~0\longrightarrow \overline IT \longrightarrow T \longrightarrow T/\overline IT \longrightarrow 0. $$
$$(2)~~~~~~0\longrightarrow \overline I T \longrightarrow \overline{\mathcal R}_+(1) \longrightarrow \overline S \longrightarrow 0 $$
$$(3)~~~~~~0 \longrightarrow \overline{\mathcal R}_+(1) \longrightarrow \overline{\mathcal R} \longrightarrow \overline{\mathcal G} \longrightarrow 0, $$
$$(4)~~~~~~0 \longrightarrow \overline{\mathcal R}_+ \longrightarrow \overline{\mathcal R} \longrightarrow R \longrightarrow 0$$
of graded $T$-modules. 
\end{proof}

Applying Lemma \ref{Sallymodule} to the case where the base ring $R$ is analytically unramified Cohen-Macaulay, we get $\depth \overline{\mathcal G} \ge \depth_T \overline S -1$ (\cite[Proposition 2.4]{CPR2}(b)) and $\overline{\mathcal G}$ is Cohen-Macaulay if $S=(0)$ (\cite[Proposition 2.4]{CPR2}(a)).

\begin{rem}\label{rems3} Suppose that conditions $(C_1)$ and $(C_3)$ are satisfied. Then $T$ is a $(S_3)$-ring by using \cite[Theorem 6.2]{T2}. Therefore if we assume that conditions $(C_1)$, $(C_2)$ and $(C_3)$ are satisfied, then by Lemma \ref{Sallymodule} $\Ass_T(\overline S) \subseteq \{\m T\}$. Hence $\dim_T\overline S = d$ provided $\overline S \ne (0)$.
\end{rem}

The following lemma play a crucial role on computing the normal Hilbert polynomial in Theorem \ref{mainth}. 

\begin{lem} \label{hbcoef} Suppose that conditions $(C_0)$ and $(C_2)$ are satisfied. Then the following assertions hold true.
\begin{enumerate}[\rm (1)]
\item {\cite[Proposition 2.4]{GO}} For every $n\ge 0$ 
\begin{eqnarray*} 
\ell_R(R/\overline{I^{n+1}})&=& \e_0(I) {{n+d}\choose{d}} - \{\e_0(I) + \e_1(Q) -\ell_R(R/\overline I)\}{{n+d-1}\choose{d-1}}\\
& & + \sum_{i=2}^{d}(-1)^i\{\e_{i-1}(Q) + \e_i(Q){{n+d-i}\choose{d-i}}\} - \ell_R(\overline S_n).
\end{eqnarray*}
\item {\cite[Proposition 2.5]{GO}} $\overline{\e_1}(I) = \e_0(I) + \e_1(Q) -\ell_R(R/\overline I) + \ell_{{T}_{\m T}}(\overline S_{\m T}),$
whence $\overline{\e_1}(I) \ge \e_0(I) + \e_1(Q) -\ell_R(R/\overline I)$.
\end{enumerate}
\end{lem}

We omit the proof of the above lemma because they are the same as in \cite{GO}. Here we notice that under the condition $(C_0)$, $\ell_R(R/Q^{n+1}) = \sum_{i=0}^{d}(-1)^i\e_{i}(Q){{n+d-i}\choose{d-i}}$ for every $n\ge0$ by \cite[Theorem 4.1]{T1}. 

The following lemma shows that the equality $\overline{\e_1}(I) = \e_0(I) + \e_1(Q) - \ell_R(R/\overline{I})$ corresponds to the case where either $\overline S$ vanishes or the reduction number of the normal Hilbert filtration is at most one. And the equality $\overline{\e_1}(I) = \e_0(I) + \e_1(Q) - \ell_R(R/\overline{I}) + 1$ corresponds to the normal Sally module of rank one.

\begin{lem}\label{0and1} Assume that conditions $(C_1)$, $(C_2)$, and $(C_3)$ are satisfied. Then the following assertions hold true.
\begin{enumerate}[\rm (1)] 
\item The followings are equivalent to each other
\begin{enumerate}[\rm (a)]
\item $\overline{\e_1}(I) = \e_0(I) + \e_1(Q) - \ell_R(R/\overline{I})$
\item $\overline S = 0$. 
\item $\rmr_Q(\{\overline{I^n}\}_{n\in\mathbb Z}) \le 1$.
\end{enumerate}

When this is the case we get the followings.

\begin{enumerate}[\rm (i)]
\item $\depth \overline{\mathcal G}\ge \depth R$.
\item for all $n\ge 0$ 
\begin{eqnarray*}
\ell_R(R/\overline{I^{n+1}})&=& \e_0(I) {{n+d}\choose{d}} - \{\e_0(I) + \e_1(Q) - \ell_R(R/\overline{I})\} {{n+d-1}\choose{d-1}} \\
 & & + \sum_{i=2}^{d}(-1)^i\{\e_{i-1}(Q) + \e_i(Q)\}{{n+d-i}\choose{d-i}}. 
\end{eqnarray*} 
\end{enumerate}

\item {\cite[Theorem 2.9]{GO}} $\overline{\e_1}(I) = \e_0(I) + \e_1(Q) - \ell_R(R/\overline{I}) +1$ if and only if $\m\overline S = 0$ and $\rank_B\overline{S} = 1$.
\end{enumerate}
\end{lem}

\begin{proof} The statement (2) is by \cite[Theorem 2.9]{GO}. Now we will prove (1). Since conditions $(C_1)$, $(C_2)$, and $(C_3)$ are satisfied, $\Ass_{T}\overline S \subseteq \{\m T\}$ by Lemma \ref{Aux} and we get $\overline S_{\m T}=(0)$ if and only if $\overline S = (0)$. Moreover by Lemma \ref{hbcoef}(2) $\overline{\e_1}(I) = \e_0(I) + \e_1(Q) -\ell_R(R/\overline I) + \ell_{{T}_{\m T}}(\overline S_{\m T}),$ therefore $\overline{\e_1}(I) = \e_0(I) + \e_1(Q) - \ell_R(R/\overline{I})$ if and only if $\overline S = 0$. The equivalence of $(b)$ and $(c)$ is clear. Now assume that $\overline S = (0)$. The statement $(i)$ follows by Lemma \ref{Sallymodule}(3). The last assertion $(ii)$ follows by Lemma \ref{hbcoef}(1).
\end{proof}

\section{Proof of Theorem \ref{mainth} and Theorem \ref{e2}}

This section is devoted for presenting the proofs of Theorem \ref{mainth} and Theorem \ref{e2}. In order to do this we need the result that the integral closure of $\rmR(\mathcal I)$ in $R[t]$ is a $(S_2)$-ring. Here we notice that a Noetherian ring $R$ is called Nagata if for every $P\in \Spec R$, for any finite extension $L$ of $\mathscr Q(R/P)$, the integral closure of $R/P$ in $L$ is a finite $R/P$-module, where $\mathscr Q(R/P)$ denotes the quotient field of $R/P$ (see \cite[31.A DEFINITIONS]{Mat}). Let $\mathcal I = \{I_n\}_{n\in \mathbb Z}$ be a filtrations of ideals in $R$, that is $I_n$ is an ideal of $R$ for every $n\in \mathbb Z$, $I_0 = R$, $I_n\supseteq I_{n+1}$ for every $n\in \mathbb Z$, and $I_mI_n\subseteq I_{mn}$ for all $m,n\in \mathbb Z$. Then we get the following result which is belong to Shiro Goto. 

\begin{prop}\label{S2} Assume that $R$ be a Noetherian local ring with the maximal ideal $\m$, $d=\dim R >0$ such that $R$ is a reduced, Nagata, and $(S_2)$-ring. Let $\mathcal I = \{I_n\}_{n\in \mathbb Z}$ be a filtrations of ideals in $R$ such that $I_1 \ne R$. Suppose that $\height_R I_1 \ge 2$, and $\rm R(\mathcal I)\subseteq R[t]$ is Noetherian. Then the integral closure of $\rm R(\mathcal I)$ in $R[t]$ is a $(S_2)$-ring.
\end{prop}
\begin{proof}  
We denote $\mathscr Q(-)$ the quotient field of $(-)$. Put $\mathscr R:= \rm R(\mathcal I)$ and $\mathscr F = \mathscr Q(R[t])$. Let $\mathscr S$ and $\mathscr T$ denote the integral closure of $\mathscr R$ in $R[t]$ and $\mathscr F$, respectively. Since $\height_R I_1 \ge 2$, there exists an $R$-regular element $a\in I_1$. Put $f=at$. Then $\mathscr Q(R[f]) = \mathscr Q(R[t])\supseteq \mathscr Q(R)$. Let $\overline R$ be the integral closure of $R$ in $\mathscr Q(R)$. Then $\overline R$ is a finitely generated graded $R$-module. Since $\overline R[t]$ is integral closed in $\mathscr F$, $\mathscr T\subseteq \overline R[t]$. Therefore $\mathscr S\subseteq \mathscr T$ and $\mathscr S=\mathscr T\cap R[t]$. Since $R$ is Nagata and $\mathscr R$ is Noetherian, $\mathscr T$ is a finitely generated graded $\mathscr R$-module and hence $\mathscr S$ is Noetherian. Assume on the contrary that $\mathscr S$ is not a $(S_2)$-ring. Then there is a prime ideal $P$ of $\mathscr S$ such that 
$$\depth \mathscr S_P <\inf\{2, \dim \mathscr S_P\}.$$
If $\dim \mathscr S_P\le 1$, then $\depth \mathscr S_P = 0, \dim \mathscr S_P = 1$, which contradicts to the fact that $\mathscr S$ is reduced. Therefore $\dim \mathscr S_P\ge 2$ and $\depth \mathscr S_P = 1$. Moreover this ideal $P$ is graded. In fact, assume on the contrary that $P$ is not graded. Then $P\ne P^{*}$, where $P^{*}$ denotes the graded ideal generated by all homogeneous elements of $P$. Therefore $\depth \mathscr S_{P^{*}} = 0$. Furthermore since $\mathscr S$ is reduced, $\depth \mathscr S_{P^{*}} = \dim \mathscr S_{P^{*}}$ and from this we get $\dim \mathscr S_P = 1$, which is a contradiction. Now we put $p = P\cap R$. 

\begin{claim}\label{claimI1} $p\supseteq I_1$.
\end{claim}
\begin{proof}[Proof of Claim \ref{claimI1}] Assume on the contrary that $p\not\supseteq I_1$. Then $\mathscr R_p = \mathscr S_p = R_p[t]$. Thanks to the embedding $R_p\hookrightarrow \mathscr S_p$ and the fact that $\depth \mathscr S_P = 1$  we have $\depth R_p\le 1$. Since $R$ is $(S_2)$, $\depth R_p\ge\inf \{ 2,\dim R_p\}$ and we have $R_p$ is Cohen-Macaulay and so is $\mathscr S_p$. Since $\mathscr S_P = (\mathscr S_p)_{P\mathscr S_p}$, $\mathscr S_P$ is Cohen-Macaulay, which is impossible. Thus $p\supseteq I_1$ as wanted. 
\end{proof}

\begin{claim}\label{claims2} For all graded prime ideal $\p$ of $\mathscr S$ such that $\height \p \ge 2$, we have $\height_{\mathscr T} \P\ge 2$  for all prime ideal $\P$ of $\mathscr T$ such that $\P\cap \mathscr S=\p$. 
\end{claim}

\begin{proof}[Proof of Claim \ref{claims2}] Take $\P_0\in \Min \mathscr T$ such that $\P_0\subseteq \P$ and $\dim \mathscr T_{\P} = \dim \mathscr T_{\P}/{\P}_0\mathscr T_{\P}$. 
Since $\P_0 \in \Ass_{\mathscr T}\overline R[t]$, there is $U\in \Ass_{\overline R[t]}\overline R[t]$ such that $U\cap \mathscr T = \P_0$. Put $W=\P_0\cap \overline R$. Then $W\in \Ass_{\overline R}\overline R$ and $U=W\overline R[t]$. Put $p_0 = W\cap R$. Then $p_0\in\Ass R$ and therefore $\height_R p_0=0$ as $R$ is $(S_1)$. Since 
$$\P_0\cap \mathscr R = U\cap \mathscr R = W\overline R[t]\cap \mathscr R = (W\overline R[t]\cap R[t])\cap \mathscr R = p_0R[t]\cap \mathscr R,$$ 
we get $(\P_0\cap \mathscr R)\cap R = p_0$ and $\mathscr R/(\P_0\cap \mathscr R) \cong \mathscr R(\{(I_n+p_0)/p_0\})$. Therefore $\dim \mathscr R/(\P_0\cap \mathscr R) = d+1$, where $d=\dim R\ge 2$. Similarly we also get $\mathscr S/(\P_0\cap \mathscr S) \cong \mathscr R(\{(\overline{I_n}+p_0)/p_0\})$ and hence $\dim \mathscr S/(\P_0\cap \mathscr S) = d+1$. Now let $M$ be the graded maximal of $\mathscr S$. Then $\p\subseteq M$. Since the extension $\mathscr S\subseteq \mathscr T$ is finite, by Going Up theorem, there exists a graded maximal ideal $N$ of $\mathscr T$ such that $\dim \mathscr T_N/\P\mathscr T_N = \dim \mathscr S_M/\p\mathscr S_M =:\alpha$. 
\begin{displaymath}
    \xymatrix{
		\P_0\ar@{-}[r]& \P \ar@{-}[r]^{\alpha}& ^{\exists}N\subseteq \mathscr T\\
		{\P_0\cap \mathscr S} \ar@{-}[u] \ar@{-}[r]& \p\ar@{-}[u]\ar@{-}[r]^{\alpha}&M\subseteq  \mathscr S\ar@{-}[u]\\
		{\P_0\cap \mathscr R} \ar@{-}[u] \ar@{-}[r]& \P\cap \mathscr R \ar@{-}[r]& \mathscr R\ar@{-}[u] }
\end{displaymath}
Since the extension $({\mathscr S/(\P_0\cap \mathscr S)})_M\subseteq ({\mathscr T/\P_0})_N$ is finite and $R$ is universal catenary, $(\mathscr S/(\P_0\cap \mathscr S))_M$ is universal catenary local domain. Therefore $\height N/\P_0 = \height M/(\P_0\cap \mathscr S)$ and hence $\dim \mathscr T/\P_0 = d+1$. Now we assume on the contrary that $\dim \mathscr T_{\P}\le 1$. Then $\alpha \ge d$. On the other hand, since $\height \p \ge 2$, $d+1 = \alpha + \height \p\ge \alpha + 2$. Hence $\alpha \le d-1$ which yields a contradiction. Thus $\dim \mathscr T_{\P}\ge 2$.
\end{proof}
Therefore $\depth_{\mathscr S_P} \mathscr T_P\ge 2$ because of the following fact which we omit the proof.  

\begin{claim}\label{depth} Let $A$ be a Noetherian local ring with the maximal ideal $\a$ and $B$ is a finite extension of $A$ with $(S_2)$ property. If for every maximal ideal $\b$ of $B$, $\dim B_{\b}\ge 2$, then $\depth_A B \ge 2$.   
\end{claim}

Next we consider the exact sequence
$$0\longrightarrow \mathscr S_P \longrightarrow \mathscr T_P\longrightarrow (\mathscr T/\mathscr S)_P\longrightarrow 0$$
of graded $\mathscr S_P$-modules. Here we notice that $\mathscr T/\mathscr S \ne (0)$ since $\depth_{\mathscr S_P} \mathscr T_P\ge 2$ but $\depth{\mathscr S_P} = 1$. Applying the Depth lemma to the above exact sequence we get $\depth_{\mathscr S_P} (\mathscr T/\mathscr S)_P = 0$. 
Therefore $P\in \Ass_{\mathscr S}\mathscr T/\mathscr S$ and then $P\in \Ass_{\mathscr S}\overline{R}[t]/R[t]$ since $\mathscr T\cap R[t]= \mathscr S$.  Hence $P = Q\cap \mathscr S$ for some $Q\in \Ass_{R[t]}\overline{R}[t]/R[t]$. Moreover since $\overline{R}[t]/R[t]\cong (\overline R/R)\otimes _R R[t]$ and $\Ass_{R[t]}(\overline R/R)\otimes _R R[t] = \bigcup_{p\in \Ass_R\overline R/R}\Ass_{R[t]}(R/p)\otimes_R R[t]$, there is $p_1\in \Ass_R\overline R/R$ such that $Q\in \Ass_{R[t]}R[t]/p_1R[t]$. Since $Q = p_1R[t]$, $p_1 = Q\cap R = P\cap R =p$. Therefore $p\in \Ass_R\overline R/R$. Furthermore since $p\supseteq I_1$, $\dim R_p\ge 2$ and hence $\depth R_p\ge 2$ as $R$ is $(S_2)$. On the other hand $\depth_{R_p} (\overline R)_p >0$ as there is an $R$-regular element in $R$ which is also $\overline R$-regular. Now applying Depth lemma to the following exact sequence 
$$0\longrightarrow R_p \longrightarrow (\overline R)_p \longrightarrow (\overline R/R)_p \longrightarrow 0$$
of $R_p$-modules we get a contradiction. Thus $\mathscr S$ is an $(S_2)$-ring.
\end{proof}

Now it is a position to prove Theorem \ref{mainth}. 

\begin{proof}[Proof of Theorem \ref{mainth}] Since conditions $(C_1)$, $(C_2)$, and $(C_3)$ are satisfied, we get $\overline{\e_1}(I) = \e_0(I) + \e_1(Q) -\ell_R(R/\overline I) + \ell_{{T}_{\m T}}(\overline S_{\m T})$, by Lemma \ref{hbcoef}(2)  and $\Ass_T(\overline S) \subseteq \{\m T\}$, by Lemma \ref{Sallymodule}(2).
 
$(3) \Rightarrow (2)$ This is obvious.
 
$(2) \Leftrightarrow (1)$ This is by Lemma \ref{0and1}(2).

$(1) \Rightarrow (3)$ Assume that $\overline{\e_1}(I) = \e_0(I) + \e_1(Q) -\ell_R(R/\overline I)+1$. Then $\overline S \ne (0)$ by Lemma \ref{0and1} and hence $\Ass_T\overline S = \{\m T\}$. Therefore $\overline S$ is a torsion free $B$-module. If $d=1$, then $B$ is a PID. Hence $\overline S$ is $B$-free because every torsion free modules over a PID are free.
Now we consider the case where $d\ge 2$. We will show that $\overline S$ is a $(S_2)$ module over $B$. When this is the case, since $\rank_B \overline S = 1$ and $B$ is an UFD, $\overline S$ is a reflexive $B$-module and hence a free $B$-module. Therefore $\overline S/B_+\overline S = (R/\m)\overline{\varphi}$ for some homogeneous element $\varphi \in (\overline S)_q$ of degree $q\ge 1$. Hence $\overline S = B\varphi + B_+\overline S$ and $(\overline S)_{B_+} = B_{B_+}\frac{\varphi}{1}$ by the graded Nakayama lemma. So $(\overline S/B\varphi)_{B_+} =0$ and $\overline S/B\varphi =0$. Thus $\overline S \cong B(-q)$ as desired.   

Now we assume on the contrary that $\overline S$ is not a $(S_2)$ module over $B$. Then 
$$\depth_{B_P} \overline{S}_P < \inf\{2, \dim_{B_P} \overline S_P\}$$ for some prime ideal $P\in\Supp_B \overline S$. Therefore $\dim_{B_P} \overline S_P \ge 2$ and $\depth_{B_P} \overline S_P = 1$. Here we notice that $\dim_{B_P} \overline S_P = \dim B_P$. This ideal $P$ is a graded ideal of $B$. In fact, assume on the contrary that $P$ is not graded. Then $1 = \depth_{B_P} \overline S_P = \depth_{B_{P^{*}}} \overline S_{P^{*}} +1$, where the graded ideal $P^{*}$ denotes the ideal generated by all homogeneous elements in $P$. Therefore we have $\depth_{B_{P^{*}}}\overline S_{P^{*}} = 0$ and hence $P^{*}\in \Ass_B(\overline S)$. Thus $\height P^{*}=0$. 
On the other hand, we have $2\le \dim B_P = \dim_{B_{P}} \overline S_P = \dim_{B_{P^{*}}} \overline S_{P^{*}} +1 = \dim B_{P^{*}} +1$. 
This means $\height P^{*}\ge 1$ which is a contradiction. Thus $P$ is graded. Now let $p\in \Spec T$ such that $P=p+\m T$. Then $p$ is also graded as $\m T$ is graded. Moreover $\height_{T} p \ge 3$ because $\height_B P \ge 2$, $\m T \subseteq p$ and $\height_T \m T =1$. We will prove that $\depth_{T_p} (\overline{\mathcal R})_p\ge 2$. In order to prove this, it is enough to show the following.

\begin{claim} \label{for S2 1} For all graded prime ideal $\p$ of $T$ such that $\height \p \ge 3$, we have $\height_{\overline{\mathcal R}}\P\ge 2$ for all prime ideal $\P$ in $\overline{\mathcal R}$ with $\P\cap T =\p$.
\end{claim}

When this Claim \ref{for S2 1} holds true, since $\overline{\mathcal R}$ is a $(S_2)$-ring by Proposition \ref{S2}, we get $\depth_{T_{\p}} (\overline{\mathcal R})_{\p}\ge 2$ by applying Claim \ref{depth}.

\begin{proof}[proof of Claim \ref{for S2 1}]Assume on the contrary that there exists a prime ideal $\P$ of $\overline{\mathcal R}$ such that $\P\cap T = {\p}$ but $\height_{\overline{\mathcal R}}\P \le 1$. 
Take $\P_0 \in \Min{\overline{\mathcal R}}$ such that $\P_0\subseteq \P$ and $\dim {\overline{\mathcal R}}_{\P} = \dim {\overline{\mathcal R}}_{\P}/\P_0{\overline{\mathcal R}}_{\P}$. 
Then $\P_0 = qR[t]\cap \overline{\mathcal R}$, for some $q\in \Min R$, and $\P_0\cap T = qR[t]\cap T$. Therefore $(\P_0\cap T) \cap R = q$ and hence $T/(\P_0\cap T) \cong \mathcal R((q+\P)/q)$. Since $I\not\subseteq q$, $\dim T/(\P_0\cap T) = \dim \mathcal R((q+I)/q) = \dim R/q +1= d+1$ and we get $\P_0\cap T \in \Min T$.
Let $\mathcal M$ be the unique graded maximal ideals of $T$. 
Since the extension $T\subseteq \overline{\mathcal R}$ is finite, there is a graded maximal ideals $\mathcal N$ of $\overline{\mathcal R}$ such that $\dim \overline{\mathcal R}_{\mathcal N}/\P\overline{\mathcal R}_{\mathcal N} = \dim T_{\mathcal M}/ {\p}T_{\mathcal M} =\alpha$. 
\begin{displaymath}
    \xymatrix{
		\P_0\ar@{-}[r]& \P \ar@{-}[r]^{\alpha}& ^{\exists}\mathcal N \subseteq \overline{\mathcal R}\\
		{\P_0\cap T} \ar@{-}[u] \ar@{-}[r]& {\P\cap T={\p}}\ar@{-}[u]\ar@{-}[r]^{\alpha}& \mathcal M \subseteq T\ar@{-}[u] }	
\end{displaymath}
We notice that ${\p}\subseteq \mathcal N$, as ${\p}$ is graded, and $\height \mathcal N/\P_0 = \height \mathcal M/(\P_0\cap T) = d+1$, because $R$ is universal catenary and the extension $T/(\P_0\cap T) \hookrightarrow \overline{\mathcal R}/\P_0$ is finite. Therefore 
$d+1=\height \mathcal N/\P_0 \le 1+\alpha$ and  $d+1=\height \mathcal M/(\P_0\cap T)\ge 3 + \alpha,$ a contradiction.
\end{proof}

\begin{claim} \label{for S2 2} $\depth_{T_p} (\overline{\mathcal R}_+)_p\ge 2$.
\end{claim}
\begin{proof}[proof of Claim \ref{for S2 2}] We consider the following exact sequence of $T_p$-modules.
$$0\longrightarrow (\overline{\mathcal R}_+)_p \longrightarrow (\overline{\mathcal R})_p \longrightarrow (\overline R)_p \longrightarrow 0.$$
If $R_p=(0)$, then $\depth_{T_p} (\overline{\mathcal R}_+)_p = \depth_{T_p} (\overline{\mathcal R})_p\ge 2$ by Claim \ref{for S2 1}. If $R_p\ne (0)$, then $p=\mathcal M$, the unique graded maximal ideal of $T$, because $p$ is a graded ideal. Hence $\depth_{T_{\mathcal M}} R_{\mathcal M} = \depth R \ge 1$ by the condition $(C_3)$. Now by using Claim \ref{for S2 1} and applying Depth lemma to the above exact sequence we get the desired result.
\end{proof}
By Claim \ref{for S2 1}, Claim \ref{for S2 2}, and comparing depths of $T_p$-modules in the following exact sequences 
$$0\longrightarrow (\overline I T)_p \longrightarrow (\overline{\mathcal R}_+(1))_p \longrightarrow (\overline S)_p \longrightarrow 0$$
and 
$$0\longrightarrow (\overline I T)_p \longrightarrow T_p \longrightarrow T_p/(\overline I T)_p \longrightarrow 0$$
of $T_p$-modules, we get $\depth T_p =2$. On the other hand  since $T$ is an $(S_3)$-ring, by Remark \ref{rems3}, and $\height_{T} p\ge 3$, we get $2=\depth T_p \ge \inf\{3, \dim T_p\} = 3$ which is a contradiction. Thus $\overline S$ is a $(S_2)$ module over $B$ as desired.

The statement $(b)$ is by Lemma \ref{Sallymodule}(3).
Lastly we will prove the statement $(c)$. Since $\overline{S}\cong B(-q)$ for some integer number $q\ge 1$, $\ell_R(\overline S_n) = \ell_R(B_{n-q})$ for all $n\ge 0$. 
If $n< q$ then $\ell_R(\overline S_n) = \ell_R(B_{n-q}) = (0)$ and hence 
\begin{eqnarray*}
\ell_R(R/\overline{I^{n+1}})&=& \e_0(I) {{n+d}\choose{d}} - \{\e_0(I) + \e_1(Q) -\ell_R(R/\overline I)\}{{n+d-1}\choose{d-1}}\\
& & + \sum_{i=2}^{d}(-1)^i\{\e_{i-1}(Q) + \e_i(Q)\}{{n+d-i}\choose{d-i}}
\end{eqnarray*}
by Lemma \ref{hbcoef}(1).
If $n\ge q$ then 
$$\ell_R(B_{n-q}) = {{n-q+d-1}\choose{d-1}} = \sum_{i=0}^{q}(-1)^i{{q}\choose{i}}{{n+d-1-i}\choose{d-1-i}}.$$
Therefore
\begin{eqnarray*}
\ell_R(R/\overline{I^{n+1}})&=& \e_0(I) {{n+d}\choose{d}} - \{\e_0(I) + \e_1(Q) -\ell_R(R/\overline I)\}{{n+d-1}\choose{d-1}}\\
& & + \sum_{i=2}^{d}(-1)^i\{\e_{i-1}(Q) + \e_i(Q)\}{{n+d-i}\choose{d-i}} - \sum_{i=0}^{q}(-1)^i{{q}\choose{i}}{{n+d-1-i}\choose{d-1-i}}\\
&=& \e_0(I) {{n+d}\choose{d}} - \{\e_0(I) + \e_1(Q) -\ell_R(R/\overline I)+1\} {{n+d-1}\choose{d-1}}\\
& & + \sum_{i=2}^{d}(-1)^i\{\e_{1}(Q) + \e_2(Q) +{{q}\choose{i-1}}\}{{n+d-i}\choose{d-i}}
\end{eqnarray*}
also by Lemma \ref{hbcoef}(1). Thus the last statement $(c)$ follows. 
\end{proof}
Now we will prove Theorem \ref{e2} and examine some corollaries of Theorem \ref{mainth} in the Cohen-Macaulay case. 

\begin{proof}[Proof of Theorem {\ref{e2}}] Since $\overline{\a \widehat R} = \overline{\a}\widehat{R}$ for every $\m$-primary ideal $\a$ in $R$, by passing to the $\m$-adic completion $\widehat R$ of $R$, without lost of generality we may assume that $R$ is complete. Therefore $R$ is a Nagata reduced Cohen-Macaulay ring. Since $\overline{\e_1}(I) = \e_0(I) - \ell_R(R/\overline{I}) + 1$, we have $\depth \overline{\mathcal G} \ge d-1$ by Theorem \ref{mainth}(b). Therefore we get the assertion $(1)$. Now we will prove the assertion $(2)$. Since $\depth \overline{\mathcal G} \ge d-1$, by \cite[Proposition 4.6]{HM} the normal Hilbert coefficients have the following forms
$$\overline{\e_i}(I) = \sum_{n\geq i} {{n-1}\choose {i-1}}\ell_R(\overline{I^n}/Q\overline{I^{n-1}})$$
for $1\leq i\leq d$. Now, since $\overline{\e_3}(I) = 0$, we get $\overline{I^{n}} = Q\overline{I^{n-1}}$ for all $n\ge 3$ and therefore $\overline{\mathcal G}$ is Cohen-Macaulay, by \cite[Theorem 4.6(ii)]{H2}, because $\overline{I^2}\cap Q = Q\overline I$ by \cite[THEOREM 1]{I1}. Since $\overline{\e_1}(I) = \sum_{n\ge 1}\ell_R(\overline{I^n}/Q\overline{I^{n-1}})$ and $\overline{\e_2}(I) = \sum_{n\ge 2}(n-1)\ell_R(\overline{I^n}/Q\overline{I^{n-1}})$, we have $\overline{\e_1}(I) = \ell_R(\overline{I}/Q) + \ell_R(\overline{I^2}/Q\overline{I})$ and $\overline{\e_2}(I) = \ell_R(\overline{I^2}/Q\overline{I})$. Therefore $\overline{\e_2}(I) = \ell_R(\overline{I^2}/Q\overline{I}) = 1$. The last statement follows from the results that $\ell_R(\overline{I^2}/Q\overline{I}) = 1$ and $\overline{I^{n}} = Q\overline{I^{n-1}}$ for all $n\ge 3$.
\end{proof}


When $R$ is a Nagata reduced Cohen-Macaulay ring, with a further assumption that $\overline{\e_3}(I) = 0$, the number $q$ in Theorem \ref{mainth} turns into exactly $1$ and $\overline{\mathcal G}$ and $\overline{\mathcal R}$ are Cohen-Macaulay as follows. 

\begin{cor}\label{cm} Assume that $R$ be a Nagata reduced Cohen-Macaulay local ring with the maximal ideal $\m$ and $d=\dim R \ge 3$. Let $I$ be an $\m$-primary ideal of $R$ and suppose that $I$ contains a parameter ideal $Q$ of $R$ as a reduction. Assume that $\overline{\e_3}(I) = 0$. Then the followings are equivalent to each other.
\begin{enumerate}[\rm (1)]
\item $\overline{\e_1}(I) = \e_0(I) - \ell_R(R/\overline{I}) + 1$.
\item $\m\overline{S} = (0)$ and $\rank_B\overline{S} = 1$, where $B = T/\m T$.
\item $\overline S \cong B(-1)$ as graded $T$-modules.
\item $\overline{\e_2}(I) = 1$ and $\ell_R(\overline{I^2}/Q\overline{I}) = 1$.
\end{enumerate}
When this is the case 
\begin{enumerate}[\rm (a)]
\item $\overline S $ is a Cohen-Macaulay $T$-module.
\item $\overline{\mathcal G}$ is Cohen-Macaulay.
\item $\overline{\mathcal R}$ is Cohen-Macaulay.
\item $\rm r_Q(\{\overline{I^n}\}_{n\in \mathbb Z}) =2$.
\item for all $n\ge 1$ 
\end{enumerate}
\begin{eqnarray*}
\ell_R(R/\overline{I^{n+1}})&=& \e_0(I) {{n+d}\choose{d}} - \overline{\e_1}(I) {{n+d-1}\choose{d-1}} + {{n+d-2}\choose{d-2}}.
\end{eqnarray*} 
\end{cor}

\begin{proof} The equivalent statements $(1) \Leftrightarrow (2)$ is exactly as in Theorem \ref{mainth}. By \cite[Proposition 4.9]{MMV} we have $\overline{\e_2}(I) = 1$ implies $\overline{\e_1}(I) = \e_0(I) - \ell_R(R/\overline{I}) + 1$ and then we get the implication $(4) \Rightarrow (1)$. The implication $(1) \Rightarrow (4)$ is followed by Theorem \ref{e2} and its proof. Again by Theorem \ref{mainth} we get the implication $(3)\Rightarrow (1)$. Assume that we have condition (1), then by Theorem \ref{mainth}(c) we have $\overline S \cong B(-q)$ and $\overline{\e_2}(I) = q$ for some $q\ge 1$. Moreover since $\overline{\e_2}(I) = 1$ by the implication $(1)\Rightarrow (4)$, we get the implication $(1)\Rightarrow (3)$.

Now we assume that $\overline{\e_1}(I) = \e_0(I) - \ell_R(R/\overline{I}) + 1$ and $\overline{\e_3}(I) = 0$. Then $(a)$ follows by Theorem \ref{mainth} (a). The statements $(e)$ follows by Theorem \ref{mainth} (c) and the fact that $\overline{\e_2}(I)=1$. 
$(b)$ and $(d)$ are by Theorem \ref{e2} (2). 
For a proof of $(c)$, 
we use \cite[Proposition 4.10]{HM} which said that $\overline{\mathcal R}$ is Cohen-Macaulay if and only if $\overline{\e_1}(I) = \sum_{n= 1}^{d-1}\ell_R(\overline{I^n} +Q/Q)$. Since $\overline{I^{n}} = Q\overline{I^{n-1}}$ for all $n\ge 3$, by $(d)$ and $\ell_R(\overline{I^2} +Q/Q) = \ell_R(\overline{I^2}/Q\overline I) = 1$ we get $\overline{\mathcal R}$ is Cohen-Macaulay as desired. 
\end{proof}


We end this research by giving some remarks of Theorem \ref{e2} on the Cohen-Macaulayness of $\overline{\mathcal G}$ and $\overline{\mathcal R}$ in the case where $d\le 2$. In the case where $d=2$, we do not have any information about the Cohen-Macaulayness of $\overline{\mathcal G}$. However $\overline{\mathcal R}$ may not be Cohen-Macaulay as showing in the following example.
\begin{ex}
Let $S = k[x,y]$ be the polynomial ring over a field $k$ and $A = k[x^2, xy, xy^2]~(\subseteq S)$. We set $R = A_M$ where $M = (x^2,xy,xy^2)A$ and let $\m$ denote the maximal ideal of $R$. We then have the following.
\begin{enumerate}[$(1)$]
\item $R$ is a Gorenstein local integral domain such that $\dim R = 2$, $\rme_0(\m) = 3$, and $\m^3 = Q\m^2$ where $Q = (x^2-xy^2, xy)$, but $\ell_R(\m^2/Q\m) = 1$.\item $R$ is not a normal ring but $\m$ is a normal ideal in $R$.
\item $\ell_R(R/\m^{n+1}) = 3\binom{n+2}{2} - 3\binom{n+1}{1} + 1$ for all $n \ge 0$.
\item $\overline{S}_Q(\m) = \rmS_Q(\m)  \cong B(-1)$ as a graded $T$-module.

\item $\mathcal G(\m)$ is a Cohen-Macaulay ring with $\rma (\mathcal G(\m)) = 0$, so that $\calR(\m)$ is not a Cohen-Macaulay ring.
\end{enumerate}
\end{ex}

\begin{proof}
Since $A \cong k[x,y,z]/(z^4 - xy^2)$ where $k[x,y,z]$ denotes the polynomial ring, we have $\dim R = 2$ and $\rme_0(\m) = 3$. It is direct to check the rest of Assertion (1). The ring $A$ is clearly not normal, whence so is $R$. To check that $\m$ is normal, one needs some computation which we leave to readers. Assertions (3) and (4) now follow from  \cite[Theorem 1.1]{GNO}. As $\mathcal G(\m) \cong k[x,y,z]/(xy^2)$, $\mathcal G(\m)$ is a Cohen-Macaulay ring with $\rma(\mathcal G(\m)) = 0$. Therefore $\calR(\m)$ is not a Cohen-Macaulay ring (see, e.g., \cite{GS}).
\end{proof}

When $d=1$, the following remark gives a note of the Cohen-Macaulayness of $\overline{\mathcal G}$ and $\overline{\calR}$.

\begin{rem} When $(R,\m)$ is a one-dimensional analytically unramified Cohen-Macaulay local ring, with the same notations for as above, $\overline{\mathcal G}$ is a Cohen-Macaulay ring and $\overline{\calR}$ is a Cohen-Macaulay ring if and only if $R$ is a discrete valuation ring. In fact, the fact that $\overline{\mathcal G}$ is a Cohen-Macaulay ring is by \cite[Proposition 3.25]{Mar}. For a proof of the second fact,  if $R$ is a discrete valuation ring, then $\overline I$ is a parameter ideal and hence $\overline{\mathcal R}$ is Cohen-Macaulay. Conversely let $a\in I$ such that $(a)\subseteq I$ as a reduction. Notice that by the module version of \cite[Theorem 1.1]{GS}, $\overline{\mathcal R}$ is Cohen-Macaulay if and only if $\overline{\mathcal G}$ is Cohen-Macaulay and $\a(\overline{\mathcal G})<0$, where $\a(\overline{\mathcal G})$ denotes the a-invariant of $\overline{\mathcal G}$(\cite[DEFINITION(3.1.4)]{GW}). Since 
$$\overline{\mathcal G}/at\overline{\mathcal G} \cong\mathcal G(\{\overline{I^n}+(a)/(a)\}_{n\in \mathbb Z}),$$ 
the associated graded ring of the filtration $\{\overline{I^n}+(a)/(a)\}_{n\in \mathbb Z}$, $\a(\overline{\mathcal G}/at\overline{\mathcal G}) = \a(\overline{\mathcal G}) + 1 \le 0$. Therefore for all $n\ge 1$, $\overline{I^n}\subseteq (a)+\overline{I^{n+1}}$. Moreover since $\overline{I^{\ell}}=a\overline{I^{\ell-1}}$ for $\ell\gg 0$, $\overline{I^n}\subseteq (a)$ for all $n\ge 1$ and hence $\overline{I^n}=a\overline{I^{n-1}}$ for all $n\ge 1$. In particular $\overline I = (a)$, whence $R$ is a discrete valuation ring by \cite[Corollary 2.5]{G}.
\end{rem}

\vspace{0.5cm}
\begin{ac} 
I would like to thank Professor Shiro Goto and Kazuho Ozeki for his valuable discussions. I also would like to thank
Pham Hung Quy for the helpful discussion throughout the Japan-Vietnam joint seminar on commutative algebra at Meiji Universiry in September 2015. 
\end{ac}

\end{document}